\documentclass[a4paper,11pt,subeqn]{article} %Add subeqn to use subequations
\usepackage{fancyhdr}
\usepackage{graphicx}
\usepackage{microtype} % Reduces overfull \hbox errors
\usepackage{mathdots} %For other diagonal dots \iddots
\usepackage[shortlabels]{enumitem}

\usepackage[backend=bibtex,citestyle=alphabetic,bibstyle=alphabetic,backref=true,url=false]{biblatex}
\usepackage{amsfonts,amssymb,amstext,amsmath} %This needs to come before hyperref for subequation links to work.
\usepackage{cool} %Contains predefined math commands such as \pderiv for partial derivatives, also contains special functions. %Possible (fixed) error with this package with partial derivatives: see http://tex.stackexchange.com/questions/54425/basic-use-of-derivative-with-cool-package-fails-with-missing-endcsname-inserte
\usepackage[amsmath,thmmarks,hyperref]{ntheorem}
\usepackage{tensor}
\usepackage{tensind} %Note this doesn't work with overbars on variables
\tensordelimiter{?}

\usepackage[bookmarksnumbered]{hyperref} %\usepackage[bookmarksnumbered,pagebackref=false]{hyperref}
\hypersetup{
    plainpages=false,       % needed if Roman numbers in frontpages
    unicode=false,          % non-Latin characters in Acrobat’s bookmarks
    pdftoolbar=true,        % show Acrobat’s toolbar?
    pdfmenubar=true,        % show Acrobat’s menu?
    pdffitwindow=false,     % window fit to page when opened
    pdfstartview={FitH},    % fits the width of the page to the window
    pdftitle={Canonical forms of Self-Adjoint Operators in Minkowski Space-time},    % title
    pdfauthor={Krishan Rajaratnam},     % author
    pdfnewwindow=true,      % links in new window
    linktoc=page,           % make page numbers the link on toc
    colorlinks=true,        % false: boxed links; true: colored links
    linkcolor=blue,         % color of internal links (change box color with linkbordercolor)
    citecolor=black,        % color of links to bibliography
    filecolor=magenta,      % color of file links
    urlcolor=cyan           % color of external links
}

\usepackage[capitalize]{cleveref} %This needs to be loaded last after hyperref, see : http://www.howtotex.com/packages/9-essential-latex-packages-everyone-should-use/
\usepackage{amsmath,amsfonts,amssymb}

\newcommand{\diag}{\operatorname{diag}}
\newcommand{\spa}[1]{\operatorname{span} \{#1\}}

\newcommand{\Ima}{\operatorname{Im}}

\newcommand{\scalprod}[2]{\left\langle #1,#2 \right\rangle}
\def\sp(#1,#2){\left\langle #1,#2 \right\rangle}
\def\bp#1{\sp(#1)}
\newcommand{\ind}{\operatorname{ind}} %index of a vector subspace w.r.t a scalar product

\DeclareFontFamily{U}{matha}{\hyphenchar\font45}
\DeclareFontShape{U}{matha}{m}{n}{
      <5> <6> <7> <8> <9> <10> gen * matha
      <10.95> matha10 <12> <14.4> <17.28> <20.74> <24.88> matha12
      }{}
\DeclareSymbolFont{matha}{U}{matha}{m}{n}
\DeclareFontFamily{U}{mathx}{\hyphenchar\font45}
\DeclareFontShape{U}{mathx}{m}{n}{
      <5> <6> <7> <8> <9> <10>
      <10.95> <12> <14.4> <17.28> <20.74> <24.88>
      mathx10
      }{}
\DeclareSymbolFont{mathx}{U}{mathx}{m}{n}

\DeclareMathSymbol{\obot}         {2}{matha}{"6B}
\DeclareMathSymbol{\bigobot}       {1}{mathx}{"CB}

\newcommand{\R}{\mathbb{R}}

\newcommand{\C}{\mathbb{C}}

\newcommand{\Fi}{\mathbb{F}}
\newcommand{\sgn}{\operatorname{sgn}}

\newcommand{\conj}[1]{\overline{#1}}

\usepackage{amsfonts,amssymb,amsmath}

\theoremnumbering{arabic}
\theoremstyle{break} %\theoremstyle{plain}
\usepackage{latexsym}
\theoremsymbol{\ensuremath{_\Box}}
\theorembodyfont{\itshape}
\theoremheaderfont{\normalfont\bfseries}
\theoremseparator{}
\newtheorem{theorem}{Theorem}[section]

\newtheorem{lemma}[theorem]{Lemma}
\theoremsymbol{\ensuremath{\diamondsuit}} %Symbolizes my results

\theoremsymbol{\ensuremath{_\Box}}

\theorembodyfont{\upshape}
\newtheorem{example}[theorem]{Example}

\newtheorem{remark}[theorem]{Remark}
\newtheorem{definition}[theorem]{Definition}

\theoremstyle{nonumberplain}
\theoremheaderfont{\scshape}
\theorembodyfont{\normalfont}
\theoremsymbol{\ensuremath{_\blacksquare}}
\usepackage{amssymb}
\newtheorem{proof}{Proof}

\qedsymbol{\ensuremath{_\blacksquare}}

\numberwithin{equation}{section} %sets equation numbers <chapter>.<section>.<index>

\newcounter{partCounter}
\newenvironment{parts}{
	\begin{list}{\bfseries{}Case \arabic{partCounter}~}{\usecounter{partCounter}}
}{
	\end{list}
}

\usepackage{autonum} % Suggested here: http://tex.stackexchange.com/questions/4728/how-do-i-number-equations-only-if-they-are-referred-to-in-the-text

\begin{document}
\pagenumbering{roman}
%-----------------------------------------------------------
\title{Canonical forms of Self-Adjoint Operators in Minkowski Space-time}
\author{Krishan Rajaratnam\footnote{Department of Applied Mathematics, University of Waterloo, Canada; e-mail: k2rajara@uwaterloo.ca}}
\date{\today} % Last Draft 3.2
\maketitle
%-----------------------------------------------------------
\begin{abstract}\centering
We classify the possible Jordan canonical forms of self-adjoint operators in Minkowski space-time (in fact in pseudo-Euclidean space, i.e. an indefinite inner product space) and we show how to obtain a Jordan canonical basis which also puts the metric in a canonical form.
\end{abstract}
%-----------------------------------------------------------
% Table of Contents
\tableofcontents
\cleardoublepage
\phantomsection		% allows hyperref to link to the correct page 

%\printglossaries
%\cleardoublepage
%\phantomsection

%% List of New Results
%\addcontentsline{toc}{section}{List of New Results}
%\textbf{List of New Results}
%\theoremlisttype{allname}
%\listtheorems{thmMy,propMy,corMy}
%\cleardoublepage
%\phantomsection

% List of Results
\addcontentsline{toc}{section}{List of Results}
\textbf{List of Results}
\theoremlisttype{allname} %all
\listtheorems{theorem,proposition,corollary}
\cleardoublepage
\phantomsection

\pagenumbering{arabic}
\newpage
%-----------------------------------------------------------
\section{Introduction}

Self-adjoint operators are ubiquitous in pseudo-Riemannian geometry and hence in relativity theory as well. Any linear operator metrically equivalent to a symmetric contravariant tensor is self-adjoint. The Ricci tensor, the Hessian of a smooth function, the shape operator associated with a pseudo-Riemannian hypersurface or an umbilic pseudo-Riemannian submanifold \cite[Definition~4.18]{barrett1983semi}, and Killing and conformal Killing tensors are all examples of such tensors. In relativity theory the energy-momentum tensor is one as well.

Most of our notions come from the theory of inner products, the difference being here that we weaken the positive definite condition on the inner product to just indefinite\footnote{We will define these notions more precisely in the next section.} (i.e. non-degenerate). We refer to such a product as a scalar product \cite{barrett1983semi}, and a vector space equipped with a scalar product a scalar product space. The definitions of self-adjoint, unitary, normal etc. hold in this more general context as well. So we say a linear operator $T$ on vector space $V$ with scalar product $\bp{\cdot, \cdot}$ is self-adjoint if $\bp{Tx,y} = \bp{x,Ty}$ for all $x,y \in V$. The problem we solve here is that of finding a simultaneous canonical form for a self-adjoint operator $T$ and the scalar product $\bp{\cdot, \cdot}$.

This problem has some equivalent formulations that are worth noting. It often appears in the literature as finding canonical forms for ``pairs of (real) symmetric matrices (bilinear forms)'' under congruence where one is assumed to be non-singular (i.e. invertible). This equivalence can be observed if we fix a scalar product $\bp{\cdot, \cdot}$ and a self-adjoint operator $T$, then define another bilinear form $[\cdot, \cdot]$ by $[x, y] = \bp{T x, y}$. Then the problem of finding a simultaneous canonical form for $T$ and $\bp{\cdot, \cdot}$ is equivalent to that of finding one for the symmetric bilinear forms $\bp{\cdot, \cdot}$ and $[\cdot, \cdot]$ (or their matrix representations under congruence). Also a symmetric bilinear form is equivalent to a quadratic form by the polarization identity; this gives another formulation of the problem.

The problem of obtaining the possible canonical forms for self-adjoint operators in Minkowski space-time and more generally in a scalar product space has been addressed many times. A solution appeared in \cite{bromwich1906quadratic} where the problem was formulated as that of finding canonical forms for pairs of quadratic forms. Also see \cite[Chapter~5]{Gohberg2005} for another more contemporary solution, historical notes and references. The solution in \cite{Gohberg2005} is for a more general setting where the scalar product is a sesquilinear form. It also contains a classification of unitary operators and much more related results. Also see \cite{Lancaster2005} and references therein; in \cite{Lancaster2005} the Kronecker canonical form is used to solve a more general problem. 

Even though the results are out there, they are relatively hard to find and often outdated. Our approach to this problem is based on O'Neil's solution given in exercises 18-19 in \cite[P.~260-261]{barrett1983semi}. Motivated by the Jordan canonical form, we develop an algorithm to find a Jordan canonical basis for a self-adjoint operator which also gives a canonical form for the scalar product. Our derivation has some advantages: it only depends on some results from Jordan form theory, we are able to prove existence and uniqueness of the canonical form independently of the corresponding results from Jordan form theory (i.e. ours results have less dependencies), and we obtain a simple algorithm to calculate the canonical forms for self-adjoint operators. The draw back is that our solution is less general than others.

\section{Preliminary Results}

\subsection{Pseudo-Euclidean Spaces}

Suppose $V$ is a vector space over a field $\Fi$ (which for us is $\R$ or $\C$). A symmetric bilinear form on $V$ is bilinear function $\bp{\cdot,\cdot} : V \times V \rightarrow \Fi$ such that $\bp{x,y} = \bp{y,x}$ for all $x,y \in V$. A symmetric bilinear form $\bp{\cdot,\cdot}$ is called non-degenerate if for a fixed $x \in V$, $\bp{x,y} = 0$ for all $y \in V$ implies $x = 0$.

Given a non-zero vector $x \in V$, it is classified as follows:

\begin{description}
	\item[timelike] If $\bp{x,x} < 0$
	\item[lightlike] If $\bp{x,x} = 0$
	\item[spacelike] If $\bp{x,x} > 0$
\end{description}

We define a scalar product on a vector space $V$ as a non-degenerate symmetric bilinear form on $V$. A vector space $V$ equipped with a scalar product is called a scalar product space. The index of a real scalar product space $V$, denoted $\ind V$, is defined as the number of timelike basis vectors in an orthogonal basis for the scalar product, which is an invariant of the scalar product by Sylvester's law of inertia. For all notions related to the index, we will assume the scalar product space is real.

The Euclidean metric given as follows is an example of a non-degenerate scalar product:

\begin{equation}
	\bp{x,y} = \sum\limits_{i=1}^{n}x_{i}y_{i}
\end{equation}

A vector space equipped with the Euclidean metric is called a Euclidean space. The standard example of a non-degenerate scalar product with non-zero index is the Minkowski metric given as follows:

\begin{equation}
	\bp{x,y} = \sum\limits_{i=2}^{n}x_{i}y_{i} - x_{1}y_{1}
\end{equation}

A vector space equipped with the Minkowski metric is called a Lorentz space (Minkowski space).

Given a subspace $H \subseteq V$, we denote the orthogonal subspace of $H$ as $H^{\perp}$ which is defined as follows:

\begin{equation}
	H^{\perp} = \{x \in V : \bp{x,y} = 0 \quad \text{for all $y \in H$} \}
\end{equation}

$H^{\perp}$ is complementary to $H$ (i.e. $V = H \oplus H^{\perp}$) iff the restriction of the scalar product to $H$ is non-degenerate \cite[P.~49]{barrett1983semi}. One can also show that for a non-degenerate subspace $H$, $\ind V = \ind H + \ind H^{\perp}$. If $U,W$ are subspaces of $V$, then $V = U \obot W$ means that $V = U \oplus W$ and $U \perp W$.

A linear operator $T$ on a scalar product space $V$ is said to be self-adjoint if $\bp{Tx,y} = \bp{x,Ty}$ for all $x,y \in V$. Note that any polynomial in $T$ is self-adjoint if $T$ is.

We denote the complexification of a real vector space $V$ by $V^{\C}$. We define the complexified bilinear form to be the symmetric bilinear form in $V^{\C}$ obtained from the real one by a linear extension. Note that the complexified bilinear form is symmetric, in contrast with the usual Hermitian form which is not symmetric. It follows immediately from the definition that the complexified bilinear form is non-degenerate iff the real bilinear form is non-degenerate. Thus a real scalar product space, $V$, can be canonically complexified to a complex scalar product space, hereafter denoted $V^{\C}$.

A set $v_{1},...,v_{n}$ for $V$ is said to be \emph{orthonormal} if $\bp{v_{i},v_{i}} = \pm 1$ and $\bp{v_{i},v_{j}} = 0$ for $i \neq j$. Clearly an orthonormal set forms a basis for $V$ and the metric in this basis is $g = \diag(\pm 1,..., \pm 1)$.

Now assume the scalar product $\bp{\cdot, \cdot}$ is (possibly) degenerate. We say a sequence of vectors $v_1, \dotsc , v_p$ is a \emph{skew-normal sequence} of (length $p$) and (sign $\varepsilon = \pm 1$) if $\bp{v_{i},v_{j}} = \varepsilon$ when $i + j = p+1$ and $\bp{v_{i},v_{j}} = 0$ otherwise. We will show shortly that these vectors are necessarily linearly independent, so let $H = \spa{v_1, \dotsc , v_p}$. Then the bilinear form restricted to $H$ is skew-diagonal and is given as follows:

\begin{equation}
		g = \begin{pmatrix}
		0 &  & \varepsilon \\ 
	 	& \iddots &  \\ 
		\varepsilon &  & 0
		\end{pmatrix}
\end{equation}

The following lemma shows that a skew-normal sequence forms a linearly independent set and it gives the index of the space spanned by such a set of vectors.

\begin{lemma} \label{lem:nullBasSig}
	Suppose $V$ is a (possibly complex) scalar product space. Suppose $\varepsilon \in \{-1,1\}$ and let $z_1,\dotsc,z_p$ be a skew-normal sequence of sign $\varepsilon$.
	
	Then $z_1,\dotsc,z_p$ form a linearly independent set and the subspace $H$ spanned by these vectors is non-degenerate and has index:

	\begin{equation}
		\ind H =
		\begin{cases}
			\lfloor \frac{p+1}{2} \rfloor & \text{if } \epsilon = -1 \\
			p- \lfloor \frac{p+1}{2} \rfloor & \text{if } \epsilon = 1
		\end{cases}
	\end{equation}
\end{lemma}
\begin{proof}
	Given $1 \leq i \leq p$, denote the additive conjugate of i by $i' = p + 1 - i$. Note that $1 \leq i' \leq p$ and $i + i' = p+1$.
	Suppose $i < \frac{p+1}{2}$ and $j > \frac{p+1}{2}$. Define vectors $v_{i}$ and $v_{j}$ as follows:
	\begin{align}
		v_{i} & = \frac{1}{\sqrt{2}}( z_{i}+ z_{i'}) \\
		v_{j} & = \frac{1}{\sqrt{2}}(z_{j} - z_{j'})
	\end{align}
	If $2i = p+1$ then let $v_{i} = z_{i}$. Now for $i < j$ suppose $i+j = p+1$, then note that $i' = j$ and $j' = i$, observe that:
	
	\begin{align}
		\bp{v_{i},v_{j}} & = \frac{1}{2} (\bp{z_{i} + z_{j} , z_{j} - z_{i}} ) \\
		& = \frac{1}{2} (\bp{z_{j}, z_{j}} + \bp{z_{i}, z_{i}} ) \\
		& = 0
	\end{align}

	\begin{align}
		\bp{v_{i},v_{i}} & = \frac{1}{2} (\bp{z_{j} + z_{i} , z_{j} + z_{i}} ) \\
		& = \frac{1}{2} (2 \bp{z_{i}, z_{j}}) \\
		& = \epsilon
	\end{align}

	\begin{align}
		\bp{v_{j},v_{j}} & = \frac{1}{2} (\bp{z_{j} - z_{i} , z_{j} - z_{i}} ) \\
		& = -\frac{1}{2} (2 \bp{z_{i}, z_{j}}) \\
		& = -\epsilon
	\end{align}
	
	Furthermore if $2i = p+1$ then $\bp{v_{i},v_{i}} = \epsilon$. Now suppose $i + j \neq p+1$, then the only way in which $\bp{v_{i},v_{j}} \neq 0$ if $i+j'=p+1$ or $i' + j = p+1$, but one can see immediately that $i+j'=p+1$ iff $i' + j = p+1$. So suppose $i+j'=p+1$. Then $p+1 = i + i' = i + p+1 - j$, hence $i = j$, in which case $\bp{v_{i},v_{j}}$ reduces to the ones examined.
	
	Thus we conclude that $\bp{v_{i},v_{i}} = \epsilon$ if $i \leq \lfloor \frac{p+1}{2} \rfloor$, $\bp{v_{i},v_{i}} = -\epsilon$ if $i > \lfloor \frac{p+1}{2} \rfloor$ and $\bp{v_{i},v_{j}} = 0$ if $i \neq j$. Thus the conclusions follow.
\end{proof}

\subsection{Finite Dimensional Operator Theory}

Given a complex scalar $\lambda$, a non-zero vector $x \in V$ is called a generalized eigenvector for T corresponding to $\lambda$ if $(T-\lambda I)^{p} x = 0$ for some positive integer p.

\begin{definition}[Generalized Eigenspaces]
	Let $T$ be a linear operator on finite-dimensional complex vector space $V$ and let $\lambda$ be an eigenvalue of $T$. The generalized eigenspace (g-space) corresponding to $\lambda$, denoted $K_{\lambda}$, is the subset of $V$ defined by:
	
	\begin{equation}
		K_{\lambda} = \{ x \in V : (T-\lambda I)^{p}(x) = 0 \: \text{for some positive integer p} \}
	\end{equation}
\end{definition}

We say a set of distinct scalars $\lambda_{1},...,\lambda_{k}$ are the spectrum of $T$ if they constitute all eigenvalues of $T$. Furthermore the kernel of an operator $T$ is denoted by $\ker T$ or $N(T)$.

The following results concerning the g-spaces of a linear operator are proven in \cite[Section~7.1]{friedberg2003linear}.

\begin{theorem} \label{thm:gSpaces}
	Let T be a linear operator on finite-dimensional complex vector space V. Suppose $\lambda$ and $\mu$ are distinct eigenvalues of T, then the following statements are true:
	
	\begin{enumerate}
		\item $K_{\lambda}$ is a non-zero T invariant subspace of V
		\item $K_{\lambda} \cap K_{\mu} = \{ 0\}$
		\item Let $U = (T- \lambda I)|_{K_{\mu}}$, then $K_{\mu}$ is $(T- \lambda I)$-invariant and $U$ is a bijection.
		\item If m is the algebraic multiplicity of $\lambda$ then $K_{\lambda} = N(T- \lambda I)^{m}$ and $\dim K_{\lambda} \leq m$.
		\item If $\lambda_{1},...,\lambda_{k}$ is the spectrum of T, then $V = \bigoplus\limits_{i=1}^{k} K_{\lambda_{i}}$
	\end{enumerate}
\end{theorem}

Hence the above theorem implies $T$ is block diagonal in a basis adapted to the g-spaces.

\begin{definition}
	Let $T$ be a linear operator on finite-dimensional complex vector space $E$ and let $x$ be a generalized eigenvector of $T$ corresponding to the eigenvalue $\lambda$. Suppose that $p$ is the smallest positive integer such that $(T-\lambda I)^{p}(x) = 0$. Then the ordered set
	
	\begin{equation}
		\{ (T-\lambda I)^{p-1}(x),(T-\lambda I)^{p-2}(x), \dotsc , (T-\lambda I)(x), x  \}
	\end{equation}
	
	is called a cycle ($p$-cycle) of generalized eigenvectors for $T$ with eigenvalue $\lambda$. $(T-\lambda I)^{p-1}(x)$ and $x$ are called the initial vector and the end vector of the cycle, respectively. We say that the cycle has length $p$ and $x$ generates a cycle ($p$-cycle) of generalized eigenvectors.
\end{definition}

We first note that the subspace spanned by a $p$-cycle has dimension $p$. We also observe that a given cycle of generalized eigenvectors generated by $x$ with eigenvalue $\lambda$ lie in $K_\lambda$. Also $T$ restricted to this cycle has the following matrix representation:

\begin{equation}
\begin{pmatrix}
\lambda & 1 &  &  &  \\ 
 & \lambda & \ddots &  & 0 \\ 
 &  & \ddots & 1 &  \\ 
 &  &  & \lambda & 1 \\ 
 & 0 &  &  & \lambda
\end{pmatrix} 
\end{equation}

We denote $U_{\lambda} = T - \lambda I$ and if $\lambda$ is fixed we remove the subscript and refer to $U_{\lambda} = U$.

%OBSOLETE
%Clearly every cycle of generalized eigenvectors with eigenvalue $\lambda$ are in $K_{\lambda}$. The following theorem shows that $K_{\lambda}$ has a basis of such cycles:
%
%\begin{theorem}
%	Let T be a linear operator on finite-dimensional complex vector space E and let $\lambda$ be an eigenvalue of T. Then there are vectors $x_{1},...,x_{q} \in K_{\lambda}$ generating cycles of generalized eigenvectors of length $p_{1},...,p_{q}$ respectively which form a basis for $K_{\lambda}$. Furthermore given a set of such cycles of generalized eigenvectors of T, they are linearly independent iff the end vectors $\{U_{\lambda_{1}}^{p_{1}-1}(x_{1}),\dotsc,U_{\lambda_{q}}^{p_{q}-1}(x_{q}) \}$ are linearly independent.
%\end{theorem}

Suppose $T$ is a real linear operator and let $\lambda$ be an eigenvalue with non-zero imaginary part. Suppose $x$ generates a cycle of generalized eigenvectors of length $p$ with eigenvalue $\lambda$ whose end vector has linearly independent real and imaginary parts. Then it follows that $\conj{x}$ generates a cycle of generalized eigenvectors of length $p$ with eigenvalue $\conj{\lambda}$ which is linearly independent of the cycle generated by $x$. We denote the real subspace generated by these vectors as $K_{\lambda \oplus \conj{\lambda}}$ and call this the \emph{real subspace} spanned by the cycle generated by $x$. If $\lambda \in \R$, then this real subspace is just $K_\lambda$.

Knowledge of the Jordan canonical form is unnecessary for our derivation. Although for readers familiar with it, note that $K_{\lambda} \simeq N_{1} \oplus \dfrac{N_{2}}{N_{1}} \oplus \dotsc \oplus \dfrac{N_{p}}{N_{p-1}}$ where $N_i = \ker U_\lambda^i$. This shows the non-uniqueness of a given Jordan canonical basis. We will use this fact to find a Jordan canonical basis for a self-adjoint operator adapted to the scalar product.

In order to prove the uniqueness of the metric-Jordan canonical form of a self-adjoint operator we will need some theory on symmetric bilinear forms. First a \emph{diagonal representation} of a symmetric bilinear form is a basis in which the matrix representation of the form is diagonal.

\begin{theorem}[Sylvester's Law of Inertia]
	For any symmetric bilinear form defined over a real vector space, the number of positive diagonal entries and negative diagonal entries in a diagonal representation is independent of the diagonal representation.
	
	For any symmetric bilinear form defined over a complex vector space, the number of non-zero diagonal entries in a diagonal representation is independent of the diagonal representation.
\end{theorem}
\begin{proof}
	For the real case, see Theorem~6.38 in \cite{friedberg2003linear} or Theorem~6.8 in \cite{jacobson2012basic}. For the complex case, see Theorem~6.6 in \cite{jacobson2012basic}
\end{proof}

\section{Existence of the metric-Jordan canonical form}

In this section we will show how to obtain the canonical form, culminating in \cref{thm:clasSelfadj}. First we need some properties of self-adjoint operators.

\begin{theorem}[Fundamental Properties of Self-Adjoint Operators]
	Suppose V is a scalar product space and T is a self-adjoint operator on V. Suppose $H \subseteq V$ is an invariant subspace of T. Then
	
	\begin{enumerate}
		\item $T(H^{\perp}) \subseteq H^{\perp}$, i.e. $H^{\perp}$ is an invariant subspace of T. \label{thm:saInvSub}
		\item $(\operatorname{ker} T)^{\perp} = \operatorname{range} T$ and $V = \operatorname{ker} T \obot \operatorname{range} T$ iff either $\operatorname{ker} T$ or $\operatorname{range} T$ is a non-degenerate subspace \label{thm:kerAnih}
		\item Any polynomial in $T$ is self-adjoint.
	\end{enumerate}
\end{theorem}
\begin{proof}
	The proofs are immediate.
\end{proof}
\begin{remark}
	The first statement of the above theorem also holds for unitary operators on $V$, as noted by O'Neil in \cite[Section~9.4]{barrett1983semi}.
\end{remark}

The idea behind obtaining the canonical forms is as follows. First suppose that $T$ is a self-adjoint operator on a scalar product space. When $E$ is a Euclidean space, one can easily diagonalize $T$ using property~\ref{thm:saInvSub} and the fact that self-adjoint operators in Euclidean space have real eigenvalues. Indeed, after one finds a single eigenvector $v$, one can use property~\ref{thm:saInvSub} to deduce that the subspace orthogonal to $v$ must be $T$-invariant. Since in Euclidean space the subspace orthogonal to $v$ must be complementary to $v$, one can repeat this procedure to find a basis of eigenvectors for $T$.

For general indefinite scalar products, our goal will be to find a cycle of generalized eigenvectors for $T$ such that they span a non-degenerate subspace. Then as in the Euclidean case, we can use property~\ref{thm:saInvSub} to inductively build a Jordan canonical basis for $T$. We will now develop a series of lemmas to show that any self-adjoint operator admits a cycle of generalized eigenvectors whose span is a non-degenerate subspace. Then we will combine these lemmas in \cref{thm:clasSelfadj} which shows how to obtain a Jordan canonical basis for $T$ which also puts the scalar product in a canonical form.

The following theorem starts us off by showing that the g-spaces of a self-adjoint operator are always non-degenerate, in fact it says even more:

\begin{lemma} \label{lem:gSpaceNonDegen}
	Suppose $V^{\C}$ is a scalar product space and T is a real self-adjoint operator on $V^{\C}$. Let $\lambda$ and $\mu$ be distinct eigenvalues of T, then $K_{\lambda} \perp K_{\mu}$, hence if $\lambda_{1},...,\lambda_{k}$ is the spectrum of T then by \cref{thm:gSpaces}, $V^{\C} = \bigobot \limits_{i=1}^{k} K_{\lambda_{i}}$.
	
	As an immediate corollary we find that each generalized eigenspace is a non-degenerate subspace.
\end{lemma}
\begin{proof}
	Suppose $x \in K_{\lambda}$ and $y \in K_{\mu}$. Suppose $U^{p}_{\lambda}(x)=0$, since $\mu \neq \lambda$ \cref{thm:gSpaces} says that $U_{\lambda}$ is a bijection when restricted to $K_{\mu}$, hence there exists a $z \in K_{\mu}$ such that $y = U_{\lambda}^{p}(z)$. Since $U_{\lambda}^{p}$ is self-adjoint, property~\ref{thm:kerAnih} implies that $\bp{x,y} = 0$.
	
%	By \cref{thm:gSpaces} we know that $\phi_{i}(T)|_{E_{\phi_{j}}}$ is an invertible linear transformation. Suppose $(\phi_{i}(T))^{p}(x)=0$, then let $\Phi_{i} = \phi_{i}(T)$ and $z = \Phi_{i}^{-p}y$. Note that any polynomial of T is self-adjoint hence $(\phi_{i}(T))^{p}$ is self-adjoint. Then
%	
%	\begin{align}
%		\bp{x,y} & = \scalprod{x}{\Phi_{i}^{p}(z)} \\
%				 & = \scalprod{\Phi_{i}^{p}(x)}{z} \\
%				 & = 0
%	\end{align}
	
	Thus $K_{\lambda} \perp K_{\mu}$. As a consequence of this and \cref{thm:gSpaces} we see that $E = K_{\lambda} \oplus K_{\lambda}^{\perp}$, hence $K_{\lambda}$ is non-degenerate.
\end{proof}

Suppose $V$ is a scalar product space and $T$ is a self-adjoint operator on $V$. Suppose $\lambda$ is an eigenvalue of $T$ and $x \in K_{\lambda}$ generates a cycle of generalized eigenvectors of $T$ of length $p$. Let $U = (T-\lambda I)$ and $v_{i} = U^{p-i} x$ for $i \in \{1,...,p\}$. Then observe that 

\begin{align}
	\scalprod{v_{i}}{v_{j}} & = \scalprod{U^{p-i} x}{U^{p-j} x} \label{eqn:reducProds} \\
	& = \scalprod{U^{2p -i -j} x}{x}
\end{align}

If $i+j \leq p$ then by property~\ref{thm:kerAnih} and the fact that $U^{p}x = 0$ the above equation implies that $\scalprod{v_{i}}{v_{j}} = 0$. If $i + j > p$ then the above equation implies that $\scalprod{v_{i}}{v_{j}}$ only depends on the sum $i+j$. Thus in a cycle of length $p$ there are only $p$ scalar products that are variable and the above equation shows us that we only need to deal with the products $\scalprod{v_{i}}{v_{p}}$. The following lemma will show that for every g-space we can always find a generator of a cycle such that $\scalprod{v_{1}}{v_{p}} \neq 0$.

\begin{lemma} \label{lem:nondgenTsel}
	Suppose $V$ is a scalar product space and $T$ is a self-adjoint operator on $V$. Fix an eigenvalue $\lambda$ of $T$ and let $U = (T-\lambda I)|_{K_{\lambda}}$. Suppose $k \geq 0$ satisfies $U^{k} \neq 0$, then there exists an $x \in K_{\lambda}$ such that $\scalprod{U^{k} x}{x} \neq 0$.
\end{lemma}
\begin{proof}
	Suppose to the contrary that $\scalprod{U^{k}(x)}{x} = 0$ for all $x \in K_{\lambda}$. Define a bilinear form $[\cdot,\cdot ] : K_{\lambda} \times K_{\lambda} \rightarrow \Fi$ by $[x,y] = \scalprod{U^{k}(x)}{y}$. Since $U^{k}$ is self-adjoint, $[\cdot,\cdot ]$ is a symmetric bilinear form. Thus by the polarization identity, it follows that for any $x,y \in K_{\lambda}$
	
	\begin{equation}
		0 = [x,y] = \scalprod{U^{k}(x)}{y}
	\end{equation}
	
	%Let $R$ be the subspace spanned by the vectors $U^{k}v_{i}$ where the $v_{i}$ satisfy $[v_{i},v_{i}]\neq 0$ and $K$ be the vectors $v_{i}$ satisfying $[v_{i},v_{i}] = 0$. Note that $R \subseteq R(U^{k})$ and $K \subseteq \ker(U^{k})$ by non-degeneracy of $\bp{\cdot,\cdot}$. Then since $K_{\lambda} = R \obot K$ and $R(U^{k}) \perp \ker(U^{k})$, we must have that $R = R(U^{k})$ and $K = \ker(U^{k})$.
	
	Now, since $U^k \neq 0$, there exists an $x \in K_{\lambda}$ such that $U^k x \neq 0$. But by \cref{lem:gSpaceNonDegen} the scalar product is non-degenerate, hence the above equation implies that $U^k x = 0$, a contradiction. Hence the conclusion holds.
\end{proof}

Assuming $\scalprod{v_{1}}{v_{p}} \neq 0$, the following proposition shows how to adapt the cycle so that any other remaining scalar products are zero.

\begin{lemma} \label{prop:adaptCyc}
	Suppose the $v_{i}$ are as defined as above for a cycle of generalized eigenvectors of $T$ generated by $x \in K_{\lambda}$. Let $H \subseteq K_{\lambda}$ be the subspace corresponding to the cycle generated by $x$. If $\scalprod{v_{1}}{v_{p}} \neq 0$, then we can choose an $x' \in H$ such that $x'$ generates a cycle of generalized eigenvectors $v'_{i} = U^{p-i} x'$ of length $p$ spanning $H$, such that $v_1,\dotsc,v_p$ forms a skew-normal sequence of sign $\sgn \scalprod{v_{1}}{v_{p}}$ if $\lambda \in \R$ or $1$ if $\lambda \in \C \setminus \R$.
\end{lemma}
\begin{proof}
	Suppose first that $\lambda \in \C \setminus \R$, then let $v'_{p} = \bp{v_{1},v_{p}}^{-\frac{1}{2}} v_{p}$ where any square root is fine. Then observe that:
	
	\begin{align}
		\bp{v'_{1},v'_{p}} & = \bp{v_{1},v_{p}}^{-1} \bp{v_{1},v_{p}} \\
		& = 1
	\end{align}
	
	If $\lambda \in \R$ then let $v'_{p} = |\bp{v_{1},v_{p}}|^{-\frac{1}{2}} v_{p}$. Then observe that:

	\begin{align}
		\bp{v'_{1},v'_{p}} & = |\bp{v_{1},v_{p}}|^{-1} \bp{v_{1},v_{p}} \\
		& = \pm 1
	\end{align}
	
	Thus we can assume that $|\bp{v_{1},v_{p}}| = 1$. Inductively suppose that $|\bp{v_{1},v_{p}}| = 1$ and that $\bp{v_{i},v_{p}} = 0$ for $2 \leq i \leq k-1$ for some $k \geq 2$.
	
	Let $v'_{p} = v_{p} + a v_{p-k+1}$ where $a$ is to be determined. Now for $i \in \{1,...,p\}$
	
	\begin{align}
		v'_{i} & = U^{p-i} v'_{p} \\
		& = U^{p-i}v_{p} + a U^{p-i}v_{p-k+1} \\
		& = v_{i} + a v_{i-k+1}
	\end{align}
	
	Observe that $v'_{i} = v_{i}$ if $i-k+1 \leq 0$, i.e. $i \leq k - 1$. The above equation also shows that each $v'_{i} \in H$ and since $v'_{1} = v_{1} \neq 0$ the cycle generated by $v'_{p}$ has length $p$ and thus forms a basis for $H$. Now using the fact that $\bp{v_{i}, v_{j}}$ only depends on $i+j$, we find that:
	
	\begin{align}
		\bp{v'_{k}, v'_{p}} & = \scalprod{v_{k} + a v_{1}}{ v_{p} + a v_{p-k+1}} \\
		& = \scalprod{v_{k}}{v_{p}} + a \scalprod{v_{k}}{v_{p-k+1}} + a \scalprod{v_{1}}{ v_{p}} + a^{2} \scalprod{ v_{1}}{ v_{p-k+1}} \\
		& = \scalprod{v_{k}}{v_{p}} + 2a \scalprod{v_{1}}{ v_{p}}
	\end{align}
	
	where $\scalprod{ v_{1}}{ v_{p-k+1}} = 0$ since $p-k+2 \leq p$. Thus let $a = - \dfrac{\scalprod{v_{k}}{v_{p}}}{2 \scalprod{v_{1}}{ v_{p}}}$ which forces $\bp{v'_{k}, v'_{p}} = 0$.
	
	Now suppose $1 \leq i < k$, then note that $v'_{i} = v_{i}$, thus
	
	\begin{align}
		\bp{v'_{i}, v'_{p}} & = \scalprod{v_{i}}{ v_{p} + a v_{p-k+1}} \\
		& = \scalprod{v_{i}}{v_{p}} + a \scalprod{v_{i}}{v_{p-k+1}} \\
		& = \scalprod{v_{i}}{v_{p}}
	\end{align}
	
	where $\scalprod{v_{i}}{v_{p-k+1}} = 0$  follows from the induction hypothesis in conjunction with the fact that because $k \geq 2$, we have that $p+i-k+1 \leq p+k-1$ and $k \leq p$ implies $p+i-k+1 \neq 1$. Thus $v'_{p}$ satisfies the induction hypothesis and after relabeling $v'_{p}$ as $v_{p}$ we can apply the induction hypothesis again until $k = p$ in which case we will have proven the statement.
\end{proof}

Suppose $x$ generates a cycle of generalized eigenvectors satisfying the conclusions of the above proposition and let $z_{i} = U^{p-i} x$. Then by \cref{eqn:reducProds} we find that the only non-zero scalar products are $\bp{z_{i},z_{j}} = \bp{z_{1},z_{p}}$ where $i+j = p + 1$. Thus we say a given cycle of generalized eigenvectors with eigenvalue $\lambda$ for a self-adjoint operator are adapted to the scalar product, if they form a skew-normal sequence of sign $\pm 1$ if $\lambda \in \R$ or sign $1$ if $\lambda \in \C \setminus \R$. If $\lambda \in \R$, then $\{z_1,\dotsc,z_p\}$ form a real basis for the Jordan canonical form of $T$. If $\lambda \in \C \setminus \R$ (WLOG we can assume $\Ima(\lambda) > 0$), then we choose a canonical real basis $\{u_{1},v_{1},...,u_{p},v_{p}\}$ for $T$ as follows. Let

\begin{subequations} \label{eq:realSeq}
	\begin{equation} \label{eq:realSeqI}
		u_i = \frac{1}{\sqrt{2}}( z_{i} + \conj{z_{i}})
	\end{equation}
	\begin{equation} \label{eq:realSeqII}
		v_i = \frac{1}{i\sqrt{2}}(z_{i} - \conj{z_{i}})
	\end{equation}
\end{subequations}

Note that

\begin{align}
	\bp{u_i , u_j } & = \frac{1}{2} (\bp{z_i , z_j} + \bp{\conj{z_{i}} , \conj{z_{j}}}) \\
	\bp{v_i , v_j } & = \frac{-1}{2} (\bp{z_i , z_j} + \bp{\conj{z_{i}} , \conj{z_{j}}}) \\
	\bp{u_i , v_j} & = \frac{1}{2 i} (\bp{z_i , z_j} - \bp{\conj{z_{i}} , \conj{z_{j}}}) = 0
\end{align}

It then follows that $\bp{u_i , u_j } = 1 = - \bp{v_i , v_j } $ if $i + j = p + 1$ with all other scalar products zero. Hence $\{u_i\}$ (resp. $\{v_i\}$) form a skew-normal sequence of sign $1$ (resp. $-1$). Now if we set $u_{p+1} = v_{p+1} = 0$, $T$ acts on this basis as follows:

\begin{align}
	T u_i  & = \frac{1}{\sqrt{2}} (\lambda z_{i} + z_{i+1} + \conj{\lambda} \conj{z}_{i} + \conj{z}_{i+1}) \\
	& = \frac{1}{\sqrt{2}} ((a+ i b) z_{i}  + (a - i b) \conj{z}_{i}) + u_{i+1} \\
	& = a u_i + \frac{b}{i\sqrt{2}} (\conj{z}_{i} - z_{i}) + u_{i+1} \\
	& = a u_i -b v_i + u_{i+1}
\end{align}

Similarly

\begin{align}
	T v_i  & = \frac{1}{i \sqrt{2}} (\lambda z_{i} + z_{i+1} - \conj{\lambda} \conj{z}_{i} - \conj{z}_{i+1}) \\
	& = \frac{1}{i\sqrt{2}} ((a+ i b) z_{i}  - (a - i b) \conj{z}_{i}) + v_{i+1} \\
	& = a v_i + \frac{b}{\sqrt{2}} (z_{i} + \conj{z}_{i}) + v_{i+1} \\
	& = a v_i + b u_i + v_{i+1}
\end{align}

In the following proposition we use these basis to show that the real subspace spanned by an adapted $p$-cycle is non-degenerate.

\begin{lemma} \label{prop:indxCyc}
	Suppose $V$ is a real scalar product space and $T$ is a self-adjoint operator on $V$. Let $x$ be a generator for a $p$-cycle of generalized eigenvectors for $T$ with eigenvalue $\lambda$ adapted to the scalar product. Let $z_{i} = U^{p-i}x$, $H$ be the real subspace spanned by this cycle and $\epsilon = \bp{z_{1},z_{p}} = \pm 1$. Then $H$ is non-degenerate.
	
	Furthermore if $\lambda \in \R$, then
	\begin{align}
		\dim H & = p \\
		\ind H & =
		\begin{cases}
			\lfloor \frac{p+1}{2} \rfloor & \text{if } \epsilon = -1 \\
			p- \lfloor \frac{p+1}{2} \rfloor & \text{if } \epsilon = 1
		\end{cases}
	\end{align}
	
	If $\lambda \in \C \setminus \R$, then
	\begin{align}
		\dim H & = 2 p \\
		\ind H & = p
	\end{align}
\end{lemma}
\begin{proof}
	If $\lambda \in \R$, then the result follows by \cref{lem:nullBasSig} applied to $z_1,\dotsc, z_p$. If $\lambda \in \C \setminus \R$, consider the real vectors $\{u_{1},v_{1},...,u_{p},v_{p}\}$ defined in \cref{eq:realSeqI,eq:realSeqII}. The result follows by \cref{lem:nullBasSig} applied to the sequence $u_1,\dotsc, u_p$ and then to $v_1,\dotsc, v_p$.
\end{proof}

The following theorem is from \cite[P.~260-261]{barrett1983semi}.
\begin{theorem}[Existence of the metric-Jordan canonical form \cite{barrett1983semi}] \label{thm:clasSelfadj}
	A linear operator T on a scalar product space V is self adjoint if and only if $V = \bigobot\limits_{i=1}^{k} V_{i}$ (hence each $V_{i}$ is non-degenerate) where each subspace $V_{i}$ is T-invariant and $T|_{V_{i}}$ has one of the following forms:
	
	\begin{equation}
	\begin{pmatrix}
	\lambda & 1 &  &  &  \\ 
	 & \lambda & \ddots &  & 0 \\ 
	 &  & \ddots & 1 &  \\ 
	 &  &  & \lambda & 1 \\ 
	 & 0 &  &  & \lambda
	\end{pmatrix} 
	\end{equation}
	
	relative to a skew-normal sequence $\{v_{1},...,v_{p}\}$ with all scalar products zero except $\bp{v_{i},v_{j}} = \varepsilon = \pm 1$ when $i+j = p+1$, or
	
	\begin{equation}
		\begin{pmatrix}
		a & b & 1 & 0 &  &  &  &  \\ 
		-b & a & 0 & 1 &  &  & 0 &  \\ 
		 &  &  &  &  &  &  &  \\ 
		 &  &  & \ddots &  &  &  &  \\ 
		 &  &  &  & a & b & 1 & 0 \\ 
		 &  &  &  & -b & a & 0 & 1 \\ 
		 & 0 &  &  &  &  & a & b \\ 
		 &  &  &  &  &  & -b & a
		\end{pmatrix}	
	\end{equation}
	
	relative to a basis $\{u_{1},v_{1},...,u_{p},v_{p}\}$ with all scalar products zero except $\bp{u_{i},u_{j}} = 1 = -\bp{v_{i},v_{j}}$ if $i+j = p+1$.
	
	The index and dimension of $V_{i}$ is determined by the blocks $T|_{V_{i}}$ due to \cref{prop:indxCyc}, hence we must have $\ind V = \sum\limits_{i=1}^{k} \ind V_{i}$ and $n = \sum\limits_{i=1}^{k} \dim V_{i}$.
\end{theorem}
\begin{proof}
	We proceed by induction. If $n=1$ then this result trivially holds. So suppose $n \geq 2$ and this result is true for all self-adjoint operators on scalar product spaces of dimension strictly less than n. Now we show that this holds when $\dim V = n$.

	Fix an eigenvalue $\lambda$ for T (which exists after complexification of $V$ if necessary). Let $U = (T-\lambda I)$. Let $p$ be the smallest integer such that $\dim N(U^{p}) = \dim N(U^{p+1})$, thus $K_{\lambda} = N(U^{p})$. Then $\dim N(U^{p-1}) < \dim N(U^{p})$, hence $U^{p-1}|_{K_{\lambda}} \neq 0$, thus by \cref{lem:nondgenTsel} there exists an $x \in K_{\lambda}$ such that $\bp{U^{p-1}x, x} \neq 0$. Note that by construction $p$ is the smallest integer such that $U^{p}x = 0$, hence x generates a $p$-cycle of generalized eigenvectors with eigenvalue $\lambda$.
	
%	By construction, p is the smallest integer such that $U^{p}x = 0$, hence by Theorem~7.6 in \cite{friedberg2003linear}, the cycle of generalized eigenvectors with eigenvalue $\lambda$ generated by x has length $p$ and is linearly independent.

	Hence by \cref{prop:adaptCyc}, the $p$-cycle of generalized eigenvectors generated by $x$ can be modified into another such $p$-cycle spanning the same subspace as the original and adapted to the scalar product. Thus we now assume that the $p$-cycle of generalized eigenvectors generated by $x$ is adapted to the scalar product. Note that it follows by \cref{lem:nullBasSig} that the set of $p$ vectors in this cycle are linearly independent. Let $H$ be the real subspace spanned by the $p$-cycle(s) generated by $x$ if $\lambda \in \R$ or by $x$ and its conjugate if $\lambda \in \C \setminus \R$. By \cref{prop:indxCyc}, $H$ is non-degenerate and by construction $H$ is $T$-invariant. If $H = V$ then we are done, so assume $H \subsetneq V$. Then by property~\ref{thm:saInvSub}, $H^{\perp}$ is an invariant subspace of $T$, and is complementary to $H$ by non-degeneracy of $H$. Let $T' = T|_{H^{\perp}}$, then $H^{\perp}$ is a scalar product space with $0 < \dim H^{\perp} < n$ and $T'$ is a self-adjoint operator on $H^{\perp}$. Hence the induction hypothesis applies to $T'$, in which case we conclude that the result holds for $T$.
	
	The converse is also easily checked.
\end{proof}

\section{Uniqueness of the metric-Jordan canonical form}

In this section $T$ is self adjoint operator on a scalar product space $V$. We will show in what sense each self-adjoint operator $T$ admits a ``unique'' metric-Jordan canonical form. We will do this by showing that the parameters appearing in any two canonical forms derived by \cref{thm:clasSelfadj} must be the same.  Then we will show how this result can be used to determine if two self-adjoint operators are isometrically equivalent.

\begin{lemma} \label{lem:cycScalP}
	Let $U = (T - \lambda I)$ for some eigenvalue $\lambda$, suppose $x$ generates an adapted cycle of length $p$ and sign $\varepsilon$ and denote by $v_i = U^{p-i} x$. Also let $H$ be the subspace spanned by this cycle.
	
	For any $0 \leq k \leq p -1$ define a symmetric bilinear form $[\cdot,\cdot]_k$ on $H$ by
	
	\begin{equation}
		[x,y]_k = \bp{U^k x, y}
	\end{equation}
	
	for $x,y \in H$. Then the number of zeros in any diagonal representation for $[\cdot,\cdot]_k$ is $k$. If the $\lambda \in \R$ then the number of negative entries in any diagonal representation for $[\cdot,\cdot]_k$ is
	
	\begin{equation}
		\begin{cases}
			\lfloor \frac{(p-k)+1}{2} \rfloor & \text{if } \epsilon = -1 \\
			(p-k)- \lfloor \frac{(p-k)+1}{2} \rfloor & \text{if } \epsilon = 1
		\end{cases}
	\end{equation}
	
	In conclusion, we see that the of invariants of $[\cdot,\cdot]_k$ depends only on $p,k,\varepsilon$.
\end{lemma}
\begin{proof}
	We prove this by exhibiting a diagonal representation for $[\cdot,\cdot]$ restricted to $H$.

	First observe that for $i,j \in \{1,\dotsc,p\}$

	\begin{equation}
		[v_i , v_j]_k = \bp{U^k U^{p-i} x, U^{p-j} x} = \bp{U^{2p+k-i-j} x,x }
	\end{equation}
	
	The above equation is non-zero iff
	
	\begin{align}
		 2p+k-i-j & = p-1 \\
		\Leftrightarrow   p+k-i-j & = -1 \\
		\Leftrightarrow i + j & = p+k+1
	\end{align}
	
	It follows that if $i < k+1$, then $[v_i , v_j] = 0$ for any $j$. Now define vectors $v'_i = v_{i+k}$ for $i \in \{1,\dotsc, p-k \}$. Then $\bp{v'_i, v'_j} \neq 0$ iff 
	
	\begin{align}
	 i+k + j+k & = p+k+1 \\
		\Leftrightarrow   i+j & = p-k+1 
	\end{align}
	
	Hence $v'_1,\dotsc,v'_{p-k}$ (or equivalently $v_{k+1},\dotsc,v_{p}$) form a pseudo-orthonormal set of vectors with sign $\varepsilon$. Thus the formula for the number of negative entries when $\lambda \in \R$ follows from \cref{lem:nullBasSig}. Also observe that the number of zeros is $k$. Then by Sylvester's law of inertia it follows that the invariants of $[\cdot,\cdot]_k$ are given as above and hence depend only on $p,k,\varepsilon$.
\end{proof}

For a real eigenvalue $\lambda$, an adapted cycle $x,Ux,\dotsc,U^{p-1}x$ is called \emph{positive} if $\bp{U^{p-1}x,x} = 1$ or negative if $\bp{U^{p-1}x,x} = -1$. By a metric-Jordan canonical basis, we mean one that is obtained from \cref{thm:clasSelfadj}.

\begin{theorem}[Uniqueness of the metric-Jordan canonical form]
	Suppose $\lambda$ is an eigenvalue of $T$. If $\lambda \in \R$,  then the number of positive (negative) cycles in $K_\lambda$ of a given length is independent of any metric-Jordan canonical basis. If $\lambda \in \C \setminus \R$,  then the number of cycles in $K_\lambda$ of a given length is independent of any metric-Jordan canonical basis.
\end{theorem}
\begin{proof}

	%Old statement showing why the $\varepsilon$ are uniquely determined
%	Now we show that the signs $\varepsilon$ are invariants. Suppose $\lambda \in \R$ and consider the basis of cycles for $K_{\lambda}$ we have just built generated by $x_{1},...,x_{l}$. Let p be the smallest integer such that $\dim N(U^{p}) = \dim N(U^{p+1})$ and suppose $x_{1},...,x_{r}$ generate cycles of length p while the remaining $x_{i}$ generate cycles of length less than p. Then note that the $\bp{U^{p-1+i}x_{j},U^{k}x_{l}} \neq 0$ iff $j = l \leq r$ and $i = k = 0$. Thus by \cref{lem:nondgenTsel} the signs $\bp{U^{p-1}x_{i},x_{i}}$ are invariants for each $i \leq r$. Similarly be restricting to an appropriate subspace we can show that the signs associated with the $x_{i}$ generating shorter cycles are also invariants.
	Fix an eigenvalue $\lambda$ of $T$ and let $U = (T - \lambda I)$. Restrict the argument to the vector space $K_\lambda$, i.e set $V = K_\lambda$. Denote by $[\cdot, \cdot]_i$ the symmetric bilinear form given by:
	
	\begin{equation}
		[x,y]_i = \bp{U^i x, y}
	\end{equation}
	
	\noindent for $x,y \in K_\lambda$. We will prove that the number of positive (negative) cycles of a given length depend only on the number of positive (negative) entries in a diagonal representation for $[\cdot, \cdot]_0,[\cdot, \cdot]_1,\dotsc,[\cdot, \cdot]_n$. It is understood that the complex representations are chosen so that there are only positive or zero entries in it. It will follow by Sylvester's law of inertia that these signs are independent of any basis.
	
	Fix a metric-Jordan canonical basis for $T|_{K_\lambda}$. It's known that $U^l = 0$ for any $l > n$, hence it follows that the number of cycles of length larger than $n$ are determined by invariants of $[\cdot, \cdot]_l$ for $l > n$. Suppose inductively that the statement holds for all cycles of length strictly larger than $p$. We will now prove the statement for cycles of length $p$.
	
	Denote by $H$ the $T$-invariant non-degenerate (possibly zero) subspace spanned by all cycles of length strictly larger than $p$ in this canonical basis. Observe that since $H$ is $T$-invariant, it follows for any $l \geq 0$ that $[x,y]_l = 0$ for $x \in H$ and $y \in H^\perp$.
	
	\begin{parts}
		\item There are no cycles of length $p$ in this canonical basis. \\
		Then note that $[x,y]_{p-1} \equiv 0$ for any $x,y \in H^\perp$ and if $H \neq 0$ the invariants of $[\cdot,\cdot]_{p-1}$ on $H$ are uniquely determined by invariants of $[\cdot,\cdot]_l$ for $l \geq p$ by \cref{lem:cycScalP}. Also the invariants of $[\cdot,\cdot]_{p-1}$ over $K_\lambda$ are determined by Sylvester's law of inertia, hence it follows that the number of cycles of length $p$ are uniquely determined.
		\item Let $x_1,\dotsc,x_m$ be generators for cycles of length $p$ in this canonical basis. \\
		For vectors from $H^\perp$ in this canonical basis the only non-zero diagonal entries of $[\cdot,\cdot]_{p-1}$ are
		
		\begin{equation}
			[x_i,x_i]_{p-1} = \bp{U^{p-1}x_i ,x_i} = \pm 1 \quad i = 1, \dotsc, m
		\end{equation}
		
		Again, if $H \neq 0$ the invariants of $[\cdot,\cdot]_{p-1}$ on $H$ are uniquely determined by invariants of $[\cdot,\cdot]_l$ for $l \geq p$ by \cref{lem:cycScalP}. The invariants of $[\cdot,\cdot]_{p-1}$ over $K_\lambda$ are determined by Sylvester's law of inertia, hence it follows that the number of positive (and negative) cycles of length $p$ are uniquely determined.
	\end{parts}
	
	Thus the result follows by induction on $p$.
\end{proof}

We can now state what we mean by ``the'' metric-Jordan canonical form:

\begin{definition}
	Let $T$ be a self-adjoint operator on a scalar product space $V$. To each adapted $p$-cycle of sign $\varepsilon$ with eigenvalue $\lambda \in \mathbb{C}$ we associate a $3$-tuple $(\lambda,p,\varepsilon)$. A canonical form given by  \cref{thm:clasSelfadj} gives an un-ordered list of such $3$-tuples counting multiplicities. We call this list \emph{the metric-Jordan canonical form}.
\end{definition}

By the above theorem, it follows that the above definition is well defined, i.e. each self-adjoint operator $T$ admits precisely one metric-Jordan canonical form. The following example shows that the signs appearing in these canonical forms add some subtleties:

\begin{example}
	Suppose $V$ is Minkowski space equipped with the standard metric 
	\begin{equation}
		g = \diag(-1,1,\dotsc,1)
	\end{equation}
	For $\lambda_1 < \dotsc < \lambda_n \in \mathbb{R}$ define two self-adjoint operators $T_1$ and $T_2$ as follows:
	
	\begin{align}
		T_1 & = \diag(\lambda_1 , \lambda_2, \lambda_3, \dotsc , \lambda_n) \\
		T_2 & = \diag(\lambda_2 , \lambda_1, \lambda_3, \dotsc , \lambda_n)
	\end{align}
	
	Now observe that even though $T_1$ and $T_2$ have the same eigenvalues, they have different metric-Jordan canonical forms. We will show shortly that $T_1$ and $T_2$ are isometrically inequivalent, in the sense that there is no $R \in O(V)$ which relates $T_1$ and $T_2$ by a similarity transformation.
\end{example}

Note that the above example is in sharp contrast with the Euclidean case where $T_1$ and $T_2$ as defined above would be isometrically equivalent.

\begin{theorem}[Isometric Equivalence of self-adjoint operators]
	Suppose $S$ and $T$ are self-adjoint operators on a scalar product space $V$. Then $S$ and $T$ differ by an isometry $R \in O(V)$ iff they have the same metric-Jordan canonical form.
\end{theorem}
\begin{proof}
	It's clear that if $S$ and $T$ have the same metric-Jordan canonical form then there is an isometry $R \in O(V)$ which relates the two operators, namely the transformation that relates a metric-Jordan canonical basis of $S$ to a metric-Jordan canonical basis of $T$.

%	Suppose $S$ and $T$ differ by an isometry $R \in O(V)$.
%	
%	\begin{equation}
%		?S^i_j? = ?R^i_k? ?S^k_l? ?R^l_j?
%	\end{equation}
	
	Suppose $T$ is given as follows relative to $S$:
	
	\begin{equation}
		T = R S R^{-1}
	\end{equation}
	
	Let $\beta = \{v_1,\dotsc,v_n\}$ be a canonical basis for $S$. Then consider the basis $\tilde{\beta} = \{R v_1,\dotsc, Rv_n \}$ for $T$. Since $R$ is an isometry, we have
	
	\begin{equation}
		g|_{\tilde{\beta}} = g|_{\beta}
	\end{equation}
	
	The equation relating $T$ to $S$ implies that
	
	\begin{equation}
		T|_{\tilde{\beta}} = S|_{\beta}
	\end{equation}
	
	Hence $S$ and $T$ have the same metric-Jordan canonical form.
\end{proof}

\section{Minkowski Space}

Fix a self-adjoint operator $T$ in Minkowski space. We will use \cref{thm:clasSelfadj} to enumerate the possible Jordan canonical forms of $T$ together with the metric in an adapted basis. As a consequence of \cref{thm:clasSelfadj}, we simply have to determine which combination of Jordan blocks are possible in Minkowski space by imposing the dimension and signature restrictions. This can be done with the help of \cref{prop:indxCyc}, since it gives us the index of a given subspace associated with a Jordan block. We denote by $D_{k}$ a diagonal $k \times k$ matrix and $I_{k}$ the identity $k \times k$  matrix. We have the following cases.

\begin{parts}
	\item $T$ is diagonalizable with real spectrum
	
	In this case $T$ must have a time-like eigenvector. Indeed, since each eigenspace $E_{\lambda}$ is non-degenerate,  one eigenspace, say $H$, must have index $1$. Then by obtaining an orthonormal basis for $H$, we can obtain a time-like eigenvector. Thus $T$ has the following form:
	
	\begin{align}
		T & = D_{n} & 
		g = \diag(-1,1,...,1)
	\end{align}

	\item $T$ has a complex eigenvalue $\lambda = a + i b$ with $b \neq 0$
	
	By \cref{prop:indxCyc} the real subspace $H$ spanned by a complex eigenvector with eigenvalue $\lambda$ and its complex conjugate must have index $1$. Since this subspace is $T$-invariant, by \cref{thm:saInvSub} $H^{\perp}$ is a complementary invariant subspace, which must be Euclidean. Hence $T$ must have the following form:
	
	\begin{align}
		T & =
		\begin{pmatrix}
		a & b & 0 \\ 
		-b & a &  \\ 
		0 &  & D_{n-2}
		\end{pmatrix}
		& 
		g =
		\begin{pmatrix}
		1 & 0 & 0 \\ 
		0 & -1 &  \\ 
		0 &  & I_{n-2}
		\end{pmatrix} 
		%\bp{e_{1},e_{1}} = - \bp{e_{2},e_{2}} = 1, \: \bp{e_{i},e_{i}} = 1 \: \text{for } i > 2
	\end{align}
	
	\item $T$ has real eigenvalues but is not diagonalizable
	
	In this case we go through the possible multidimensional Jordan blocks associated to a real irreducible subspace, say $H$, of $T$. By \cref{thm:clasSelfadj}, each basis for this subspace can be adapted to the scalar product, hence is non-degenerate. By \cref{prop:indxCyc} there are three types of Jordan blocks which have an associated subspace, $H$, with index one. For each of these subspaces, $H^{\perp}$ is a complementary $T$-invariant Euclidean subspace. The first two cases occur when $\dim H = 2$, and are given as follows:
	
	\begin{align}
		T & =
		\begin{pmatrix}
		\lambda & 1 & 0 \\ 
		0 & \lambda &  \\ 
		0 &  & D_{n-2}
		\end{pmatrix} 
		& g = 
		\begin{pmatrix}
		0 & \epsilon & 0 \\ 
		\epsilon & 0 &  \\ 
		0 &  & I_{n-2}
		\end{pmatrix} \quad \epsilon = \pm 1
		%\bp{e_{1},e_{2}} =  \epsilon = \pm 1, \: \bp{e_{i},e_{i}} = 1 \: \text{for } i > 2
	\end{align}
	
	Note that the above form contains two metric-Jordan canonical forms depending on the sign of $\epsilon$. The third occurs when $\dim H = 3$:
	
	\begin{align}
		T & =
		\begin{pmatrix}
		\lambda & 1 & 0 &  \\ 
		0 & \lambda & 1 & 0 \\ 
		0 & 0 & \lambda &  \\ 
		 & 0 &  & D_{n-3}
		\end{pmatrix}
		& g = 
		\begin{pmatrix}
		0 & 0 & 1 &  \\
		0 & 1 & 0 & 0 \\ 
		1 & 0 & 0 &  \\ 
		 & 0 &  & I_{n-3}
		\end{pmatrix} 
		%\bp{e_{1},e_{2}} =  \epsilon = \pm 1, \: \bp{e_{i},e_{i}} = 1 \: \text{for } i > 2
	\end{align}
	
	We also note that this case ($T$ has real eigenvalues but is not diagonalizable) holds iff $T$ has a unique lightlike eigenvector. This fact can be deduced by inspection of the above canonical forms.
\end{parts}

Now, we collect some necessary and sufficient conditions concerning the diagonalizability of $T$ in the following theorem. The second and third facts are from Theorem~4.1 in \cite{Hall1996} while the last fact is from Section~9.5 in \cite{greub1975linear}. All these facts can be readily deduced from the canonical forms listed above.
\begin{theorem}[Properties of self-adjoint operators in Minkowski space]
	Let V be a Minkowski space and T a self-adjoint operator on V. Then the following statements are true:
	\begin{enumerate}
		\item T is diagonalizable with a real spectrum iff T has 1 timelike eigenvector or equivalently T has $n-1$ linearly independent spacelike eigenvectors.
		\item If T has two linearly independent null eigenvectors then T is diagonalizable with a real spectrum and T has a time-like eigenspace of dimension at least 2 containing these eigenvectors.
		%if the null eigenvectors were different eigenspaces then they must be orthogonal, but this is impossible in Minkowski space.
		\item If T has a real spectrum, then T is diagonalizable iff it has no null eigenvectors or at least two linearly independent null eigenvectors. In other words, T is not diagonalizable iff it has a unique null eigendirection.
		\item If $n \geq 3$ and $\bp{Tx, x} \neq 0$ for all null vectors x then T is diagonalizable with a real spectrum.
	\end{enumerate}
\end{theorem}
%-----------------------------------------------------------
%\appendix
%\include{app1}
%-----------------------------------------------------------
%Code to add bibliography to the table of contents.
\phantomsection
\addcontentsline{toc}{section}{References}
\printbibliography
%-----------------------------------------------------------
\end{document}